\documentclass[11pt,final]{amsart}
\usepackage{amsmath, amsxtra, amssymb,xcolor}
\usepackage{verbatim}
\usepackage{graphicx}
\usepackage{multirow, tikz}

\usepackage[margin=1.55in]{geometry}
\usepackage{amssymb, mathdots}

\usepackage{graphicx}

\usepackage[cmtip,all]{xy}

\usepackage{xypic}
\usepackage[notcite, notref]{showkeys}
\usepackage[colorlinks, linkcolor=blue, citecolor=black]{hyperref}

\theoremstyle{plain}
\newtheorem{theorem}{Theorem}[section]

\newtheorem{prop}[theorem]{Proposition}
\newtheorem{corollary}[theorem]{Corollary}

\newtheorem{lemma}[theorem]{Lemma}
\newtheorem{theorem*}{Theorem}

\theoremstyle{definition}

\newtheorem{remark}[theorem]{Remark}

\numberwithin{equation}{section}

\def\ra{\rightarrow}

\newcommand\arr{\ifinner\to\else\longrightarrow\fi}

\renewcommand{\setminus}{\smallsetminus}

\newcommand{\Sym}{\operatorname{Sym}}
\newcommand{\Aut}{\operatorname{Aut}}

\newcommand{\Spec}{\operatorname{Spec}}

\newcommand{\SL}{\operatorname{SL}}

\DeclareMathOperator\spec{Spec}

\newcommand{\gitq}{/\hspace{-0.25pc}/}
\newcommand\Mg[1]{\overline{\mathcal{M}}_{#1}}
\newcommand\M{\overline{M}}

\def\co{\colon\thinspace} %macro to use in f\co \rightarrow Y

\def\co{\colon\thinspace} %macro to use in f\co \rightarrow Y 
\DeclareMathOperator{\rank}{rank}
\DeclareMathOperator{\lspan}{span}

\def\Hn1{\mathcal{H}_{n,1}}

\def\B{\mathcal{B}}

\def\O{\mathcal{O}}

\def\M{\overline{M}}

\def\I{\mathcal{I}}

\def\ZZ{\mathbb{Z}}

\def\PP{\mathbb{P}}
\def\ZZ{\mathbb{Z}}

\def\CC{\mathbb{C}}
\def\GG{\mathbb{G}}

\def\HH{\mathrm{H}}

\newcommand{\gm}{\mathbb{G}_m}

\newcommand*{\longhookrightarrow}{\ensuremath{\lhook\joinrel\relbar\joinrel\relbar\joinrel\relbar\joinrel\relbar\joinrel\relbar\joinrel\relbar\joinrel\relbar\joinrel\rightarrow}}

\begin{document}
\title[Finite Hilbert stability of curves]{Finite Hilbert stability of canonical curves, II. The even-genus case}
%\title{Finite Hilbert stability of canonical curves II:\\
%even genus case}
\author[Alper]{Jarod Alper}
\author[Fedorchuk]{Maksym Fedorchuk}
\author[Smyth]{David Ishii Smyth*}

\address[Alper]{Departamento de Matem\'aticas\\
Universidad de los Andes\\
Cra 1 No. 18A-10\\
Edificio H\\
Bogot\'a, 111711, Colombia} \email{jarod@uniandes.edu.co}

\address[Fedorchuk]{Department of Mathematics\\
Columbia University\\
2990 Broadway\\
New York, NY 10027}
\email{mfedorch@math.columbia.edu}

\address[Smyth]{Department of Mathematics\\
Harvard University\\
1 Oxford Street\\
Cambridge, MA 01238}
\email{dsmyth@math.harvard.edu}
\thanks{*The third author was partially supported by NSF grant 
DMS-0901095 during the preparation of this work.}
%\date{\today}
\maketitle

\begin{abstract}
We prove that a generic canonically embedded curve of even genus has semistable $m^{th}$ Hilbert point for all $m \geq 2$. 
More precisely, we prove that a generic canonically embedded trigonal curve of even genus has semistable $m^{th}$ Hilbert point for all $m \geq 2$. 
Furthermore, we show that the analogous result fails for bielliptic curves. 
Namely, the Hilbert points of bielliptic curves are asymptotically semistable but become non-semistable below a definite threshold value depending on $g$.
\end{abstract}

\maketitle

\setcounter{tocdepth}{1}
\tableofcontents

\section{Introduction}

This paper is a sequel to \cite{AFS-odd-stability}, where we proved that a general smooth curve of odd genus, canonically or bicanonically embedded, 
has semistable $m^{th}$ Hilbert point for all $m\geq 2$. Here, we prove an analogous result for canonically embedded curves of even genus. Our main result
is the following. 
\begin{theorem}[Main Result]\label{T:main-canonical} 
Suppose $C \subset \PP \HH^0(C, K_{C})$ is a general smooth curve of even genus, 
embedded by the complete linear system $|K_{C}|$. Then the $m^{th}$ Hilbert point of $C$ is semistable for every $m \geq 2$.
\end{theorem}
We refer to our previous paper \cite{AFS-odd-stability} for an extended discussion of the geometric motivation behind this result and its applications
 to the Hassett-Keel log minimal model program for $\M_g$, as well as
an informal description of the method of proof. As in \cite{AFS-odd-stability}, this generic stability result is obtained by proving that a very 
special singular curve has semistable Hilbert points. The singular curve we used in \cite{AFS-odd-stability} was a balanced canonical ribbon of odd genus.
The singular curve that we will use here is the so called 
{\em balanced double $A_{2k+1}$-curve} of even genus.
 
A {\em double $A_{2k+1}$-curve} is any curve obtained by gluing three copies of $\PP^1$ along two
$A_{2k+1}$ singularities (see Figure \ref{F:double-A-curve}). 
%We call such curves {\em double $A_{2k+1}$-curves}. 
In every even genus $g=2k$, double $A_{2k+1}$-curves have 
$2k-4$ moduli corresponding to the crimping of the $A_{2k+1}$-singularities, i.e., deformations that preserve the analytic type
of the singularities as well as the normalization of the curve;  we refer to \cite{fred} for a comprehensive 
discussion of crimping of curve singularities. 
We note that the parameter space of crimping for an $A_{2k+1}$-singularity with automorphism-free branches 
has dimension $k$. However, in the case of a double $A_{2k+1}$-curve, the presence of automorphisms of the pointed $\PP^1$'s 
reduces the dimension of crimping moduli by $4$.

Among double $A_{2k+1}$-curves, there is a unique double $A_{2k+1}$-curve with a $\gm$-action, 
corresponding to the trivial choice of crimping data. We call 
this curve the {\em balanced double $A_{2k+1}$-curve}.  
Our motivation for considering double $A_{2k+1}$-curves comes from 
 %the log minimal model program 
the Hassett-Keel program for $\M_{2k}$, where we expect the $2k-4$ dimensional locus 
of double $A_{2k+1}$-curves to replace the 
locus in the boundary divisor $\Delta_{k}\subset \M_{2k}$ consisting 
of curves $C_1\cup C_2$ such that each $C_i$ is a hyperelliptic curve of genus $k$. 
Indeed, this prediction has already been verified in $g=4$ by the second author who showed that 
the divisor $\Delta_2\subset \M_4$ is contracted to the point corresponding to the unique genus $4$
double $A_5$-curve in the final non-trivial log canonical model of $\M_4$ \cite{fedorchuk-genus-4}.

It is not too difficult to see that the balanced double $A_{2k+1}$-curve is trigonal,
 i.e., it lies in the closure of the locus of canonically embedded 
smooth trigonal curves; see Proposition \ref{P:A-curve}. 
From this observation, we obtain a slight strengthening of our Main Result:
\begin{theorem}[Stability of trigonal curves]\label{T:main-trigonal} 
Suppose $C \subset \PP \HH^0(C, K_{C})$ is a general smooth trigonal curve of even genus, 
embedded by the complete linear system $|K_{C}|$. Then the $m^{th}$ Hilbert point of $C$ is semistable for every $m \geq 2$.
\end{theorem}
This result leads to two related questions: 
Is it true that all smooth trigonal curves have semistable $m^{th}$ Hilbert points for all $m \geq 2$? 
Similarly, do other curves with low Clifford index have this property? 
Surprisingly, the answer to both questions is no. It is not too difficult to see that the $2^{nd}$ Hilbert point 
of a trigonal curve with a positive Maroni invariant is non-semistable; see \cite{fedorchuk-jensen} for a quick proof. In the 
final section of this paper, we will present a heuristic which suggests that a smooth trigonal curve has a semistable $m^{th}$ Hilbert point 
for $m\geq 3$. We also prove that the $m^{th}$ Hilbert point of a 
smooth bielliptic curve becomes non-semistable below 
a certain definite threshold value of $m$, depending on $g$. This is complemented by a proof of the semistability of a generic bielliptic curve of odd genus
for large values of $m$.

The outline of this paper is as follows. In Section \ref{S:A-curves}, 
we prove the basic facts about the balanced double $A_{2k+1}$-curve necessary to prove 
semistability by the strategy described in \cite{AFS-odd-stability}. In Section \ref{S:bases}, 
we construct the monomial bases necessary to prove 
semistability of the Hilbert points of the balanced double $A_{2k+1}$-curve. As a result, we obtain 
a proof of Theorems \ref{T:main-canonical} and \ref{T:main-trigonal}; see Corollary \ref{C:main-trigonal}. In Section \ref{S:bielliptic}, 
we discuss finite Hilbert stability of trigonal curves with a positive 
Maroni invariant and bielliptic curves.

We work over the field of complex numbers $\CC$.
%we consider the analogous problem for trigonal curves of higher 
%Maroni invariant and bielliptic curves.

\section{The balanced double $A_{2k+1}$-curve}
\label{S:A-curves}
%{Canonical even genus curves}
In this section, we give an explicit description of the pluricanonical linear system
 $\HH^0(C, \omega_C^m)$ of the balanced double 
$A_{2k+1}$-curve $C$. In addition, we prove the key fact 
that $\HH^0(C, \omega_C)$ is a multiplicity-free representation of $\Aut(C)$. 
Following the strategy of \cite{AFS-odd-stability}, this allows us to prove the semistability of the
 $m^{th}$ Hilbert point of $C$ by writing down monomial bases for $\HH^0(C, \omega_C^{m})$. 
In Section \ref{S:bases}, we construct the requisite monomial bases and thus prove the semistability of 
the Hilbert points of $C$.

Let us begin by giving a precise description of the balanced double $A_{2k+1}$-curve. Let $C_0, C_1, C_2$ denote three copies of 
$\PP^1$, and label the 
uniformizers at 0 (resp., at $\infty$) by $s_0, s_1,s_2$ (resp., by $t_0,t_1, t_2$). Fix an integer $k\geq 2$, and let $C$ be the arithmetic genus $g=2k$ curve obtained by gluing three $\PP^1$'s along two
$A_{2k+1}$ singularities with trivial crimping. More precisely, we impose an $A_{2k+1}$ singularity at 
$(\infty \in C_0) \sim (0 \in C_1)$ by 
gluing $C_0\setminus 0$ and $C_1\setminus \infty$ into an affine singular curve
\begin{equation}\label{E:gluing-1}
\spec \CC[x,y]/(y^2-x^{2k+2})\simeq \spec \CC[(t_0,s_1), (t_0^{k+1}, -s_1^{k+1})].
\end{equation}
Similarly, we impose an $A_{2k+1}$ singularity at $(\infty \in C_1) \sim (0 \in C_2)$
by gluing $C_1\setminus 0$ and $C_2\setminus \infty$ into %into an affine curve
\begin{equation}\label{E:gluing-2}
\spec \CC[x,y]/(y^2-x^{2k+2}) \simeq \spec \CC[(t_1,s_2), (t_1^{k+1}, -s_2^{k+1})].
\end{equation}
%and gluing them in an obvious way to obtain a double $A_{2k+1}$-curve.
%the identifications $t_0=s_1$ and $t_1=s_2$.

\begin{figure}[hbt]
\begin{centering}
\begin{tikzpicture}[scale=2]
		\node  (0) at (-3.5, 3) {};
		\node  (1) at (-2.75, 3) {};
		\node  (2) at (-2.25, 3) {};
		\node  (3) at (-1.75, 2.5) {$A_{2k+1}$};
		\node  (4) at (-3.5, 2.25) {};
		\node  (4x) at (-4, 2.25) {$A_{2k+1}$};
		\node  (5) at (-2.25, 2.25) {};
		\node  (6) at (-3, 2) {};
		\node  (7) at (-4.25, 1.5) {};
		\node  (8) at (-3.5, 1.5) {};
		\node  (8c) at (-3.7, 1.5) {$C_0$};
		\node  (9) at (-3, 1.5) {};
		\node  (10) at (-2.25, 1.5) {};
		\node  (10c) at (-2.05, 1.5) {$C_2$};
		\node  (C2) at (-2.85, 2.7) {$C_1$};
		\draw [very thick, bend left=45] (1.center) to (5.center);
		\draw [very thick] (2.center) to (10.center);
		\draw [very thick,bend left=300, looseness=1.50] (4.center) to (6.center);
		\draw [very thick] (0.center) to (8.center);
		\draw [very thick,bend left=45] (4.center) to (1.center);
		\draw [very thick,bend right=315, looseness=0.75] (5.center) to (9.center);
\end{tikzpicture}\end{centering}
\vspace{-0.5pc}
\caption{Double $A_{2k+1}$-curves}\label{F:double-A-curve}
\end{figure}
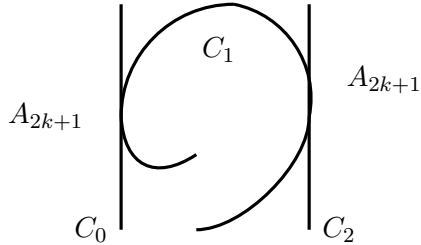

 The automorphism 
group of $C$ is given by $\Aut(C) = \GG_m \rtimes \ZZ_2$ where 
$\ZZ_2$ acts via $s_i \leftrightarrow t_{2-i}$ and 
$\GG_m = \Spec \CC[\lambda, \lambda^{-1}]$ acts via 
\begin{align*}
\lambda\cdot s_0 &= \lambda s_0, \\
\lambda\cdot s_1 &= \lambda^{-1}s_1, \\
\lambda\cdot s_2 &= \lambda s_2.
\end{align*}

Using the description of the dualizing sheaf on a singular curve as in \cite[Ch.IV]{serre-corps} or \cite[Ch.II.6]{barth},
we can write down a basis of $\HH^0(C, \omega_C)$ as follows:
\begin{equation}\label{E:basis}
\begin{aligned}
x_1 &= \left(ds_0, \frac{ds_1}{s_1^2}, 0 \right)	& y_1 &=  \left(0, ds_1, \frac{ds_2}{s_2^2} \right)  \\
x_2 &= \left(s_0 ds_0, \frac{ds_1}{s_1^3}, 0 \right)& y_2 &= \left(0, s_1ds_1, \frac{ds_2}{s_2^3} \right)  \\
	& \vdots			& & \vdots \\
x_{k} &=  \left(s_0^{k-1} ds_0, \frac{ds_1}{s_1^{k+1}}, 0 \right)& y_{k} &= \left(0, s_1^{k-1}ds_1, \frac{ds_2}{s_2^{k+1}} \right)
\end{aligned}
\end{equation}

It is straightforward to generalize this description to the spaces of pluricanonical differentials.
\begin{lemma} 
\label{L:pluricanonical-bases}
For $m\geq 2$, the product map $\Sym^m \HH^0(C,\omega_C)\ra \HH^0(C, \omega_C^m)$ is surjective and
a basis of $\HH^0(C, \omega_C^m)$ consists of the following $(2m-1)(2k-1)$ differentials:
%{\small
\begin{align*}
\omega_0 &= \left((ds_0)^m, \frac{(ds_1)^m}{s_1^{2m}}, 0 \right)	&  \eta_0 &= \left(0, (ds_1)^m, \frac{(ds_2)^m}{s_2^{2m}} \right)  \\
\omega_1 &= \left(s_0 (ds_0)^{m}, \frac{(ds_1)^m}{s_1^{2m-1}}, 0 \right)& \eta_1 &= \left(0, s_1(ds_1)^m, \frac{(ds_2)^{m}}{s_2^{2m+1}} \right)  \\
	 \vdots			& &  \vdots \\
\omega_{m(k-1)} &= \left(s_0^{m(k-1)} (ds_0)^m, \frac{(ds_1)^m}{s_1^{m(k+1)}}, 0 \right) & \eta_{m(k-1)}&= \left(0, s_1^{m(k-1)}(ds_1)^m, \frac{(ds_2)^m}{s_2^{m(k+1)}} \right)
\end{align*}%}
and
%{\small
\begin{align*}
\chi_{-k(m-1)+1} &= \left(0, s_1^{k(m-1)-m-1} (ds_1)^m, 0 \right) \\
&\vdots \\
\chi_{i} &= \left(0, s_1^{-i-m} (ds_1)^m, 0 \right) \\
& \vdots \\
\chi_{k(m-1)-1} & = \left(0, \frac{(ds_1)^m}{s_1^{(m-1)(k+1)}}, 0 \right)
\end{align*}%}
\end{lemma}
\begin{proof}
By Riemann-Roch formula, $h^0(C, \omega_C^m)=(2m-1)(2k-1)$. Thus, it suffices to observe that the given 
$(2m-1)(2k-1)$ differentials all lie in the image of the map 
$\Sym^m \HH^0(C,\omega_C) \rightarrow \HH^0(C, \omega_C^m)$. Using the basis of $\HH^0(C,\omega_C)$ given by \eqref{E:basis},
one easily checks that the differentials $\{\omega_i\}_{i=0}^{m(k-1)}$ are precisely those arising as $m$-fold products of 
$x_i$'s, the differentials $\{\eta_i\}_{i=0}^{m(k-1)}$ are those arising as $m$-fold products of $y_i$'s, and the differentials 
$\{\chi_i\}_{i=-k(m-1)+1}^{k(m-1)+1}$ are those arising as mixed $m$-fold products of $x_i$'s and $y_i$'s.
\end{proof}

Next, we show that $|\omega_C|$ is a very ample linear system, 
so that $C$ admits a canonical embedding, and the corresponding Hilbert points are well defined.

\begin{prop}\label{P:A-curve}
 $\omega_C$ is very ample. The complete linear system $|\omega_C|$ embeds $C$ as a curve on a balanced rational normal scroll 
$$\PP^1\times \PP^1 \stackrel{\vert \O_{\PP^1\times\PP^1}(1, k-1)\vert}{\longhookrightarrow} \PP^{g-1}.$$ Moreover, 
$C_0$ and $C_2$ map to $(1,0)$-curves on $\PP^1 \times \PP^1$, and $C_1$ maps to a $(1,k+1)$-curve. In particular, $C$ is a $(3,k+1)$ curve 
on $\PP^1\times \PP^1$ and has a $g^1_3$ cut out by the $(0,1)$-ruling.
\end{prop}

\begin{proof} 
To see that the canonical embedding of $C$ lies on a balanced rational normal scroll in $\PP^{2k-1}$, 
recall that the scroll % $S$\simeq \PP^1\times\PP^1$ 
can be defined as the determinantal variety (see \cite[Lecture 9]{harris}):
\begin{equation}\label{E:determinantal}
\rank
 \left(\begin{array}{cccc|cccc}
x_{1} & x_{2} & \cdots & x_{k-1} & y_{k} & y_{k-1} & \cdots & y_{2} \\  
 x_{2} & x_{3} & \cdots & x_{k} & y_{k-1} & y_{k-2} & \cdots & y_{1}
 \end{array}\right)\leq 1.
 \end{equation}
From our explicit description of the basis of $\HH^0(C, \omega_C)$ given by \eqref{E:basis}, 
one easily sees that the differentials $x_i$'s and $y_i$'s on $C$ 
satisfy the determinantal description of \eqref{E:determinantal}.
%equations for 2-by-2 minors of the above matrix. 
Moreover,
%From our explicit description of $\HH^0(C, \omega_C)$, 
we see that $\vert \omega_C\vert$ embeds $C_0$ and $C_2$ as degree $k-1$ rational normal curves in $\PP^{2k-1}$ 
lying in the class $(1,0)$ on the scroll. Also, we see that $|\omega_C|$ embeds $C_1$ via the very ample linear 
system $$\lspan\{ 1, s_1, \dots, s_1^{k-1}, s_1^{k+1}, \dots, s_1^{2k}\} \subset \vert \O_{\PP^1}(2k)\vert$$ as a curve in the class $(1,k+1)$. 
It follows that $|\omega_C|$ separates points and tangent vectors on each component of $C$. We now prove that $|\omega_C|$ 
separates points of different components and tangent vectors at the $A_{2k+1}$-singularities. %does not identify two smooth points of separate component of $C$, 
First, observe that $C_0$ and $C_2$ span different subspaces. Therefore, being $(1,0)$ 
curves, they must be distinct and non-intersecting. Second, $C_0$ and $C_1$ are the images of two branches of an $A_{2k+1}$-singularity and so have
contact of order at least $k+1$. However, being $(1,0)$ and $(1,k+1)$ curves on the scroll, they have order of contact at most $k+1$. It follows
that $C_0$ and $C_1$ on $S$ meet in a precisely $A_{2k+1}$-singularity. We conclude that $\vert \omega_C\vert$ is a closed embedding
at each $A_{2k+1}$-singularity. 

%\begin{comment}
We can also directly verify that $\vert \omega_C\vert$ separates tangent vectors at an $A_{2k+1}$ singularity of $C$, 
say the one with uniformizers $s_1$ and $t_0$. The local generator of $\omega_C$ at this singularity is
\[
x_{k}=\left(-\frac{dt_0}{t_0^{k+1}}, \frac{ds_1}{s_1^{k+1}}, 0 \right).
\] We observe that on the open affine chart $\spec \CC[(t_0, s_1), (t_0^{k+1}, -s_1^{k+1})]$ defined in Equation \eqref{E:gluing-1}
we have
$y_1=(0, s_1^{k+1}) \cdot x_{k}$ and $x_{k-1}=(t_0, s_1) \cdot x_{k}$. 
Under the identification $\CC[x,y]/(y^2-x^{2k+2})=\CC[(t_0, s_1), (t_0^{k+1}, -s_1^{k+1})]$, we have $(t_0,s_1)=x$ and $(0,s_1^{k+1})=(x^{k+1}-y)/2$.
We conclude that sections $y_1$ and $x_{k-1}$ of $\omega_C$ 
span the cotangent space $(x,y)/(x,y)^2$ and thus separate tangent vectors at the singularity $x=y=0$.

%Finally, computing the classes of $C_0$, $C_1$, and $C_2$ on the scroll is an easy exercise which we leave to the reader.
%\end{comment}
\end{proof}

Finally, the following elementary observation is the key to analyzing the stability of Hilbert points of $C$.

\begin{lemma}\label{L:multiplicityfree}  $\HH^0(C, \omega_C)$ is a multiplicity-free $\Aut(C)$-representation, i.e., no 
irreducible $\Aut(C)$-representation appears more than once in the decomposition of $\HH^0(C, \omega_C)$
into irreducibles. 
\end{lemma}

\begin{proof} Consider the basis of $\HH^0(C, \omega_C)$ given in \eqref{E:basis}.
Then $\GG_m\subset \Aut(C)$ acts on $x_i$ with weight $i$ and on $y_i$ with weight $-i$. Thus $\HH^0(C, \omega_C)$ 
decomposes into $g=2k$ distinct characters of $\GG_m$. 
\begin{comment}
On easily checks that the decomposition of $\HH^0(C, \omega_C)$ into irreducible representations is given by
$$
\HH^0(C, \omega_C)=\left(\oplus_{i=1}^{k}\CC\{x_i+y_i\}\right) \bigoplus \left(\oplus_{i=1}^{k} \CC\{x_i-y_i\}\right),
$$ 
and each of these one-dimensional representations is given by a distinct character of $\GG_m \rtimes \ZZ_2$.
\end{comment}
\end{proof}

\section{Monomial bases and semistability}
\label{S:bases}
Since $\HH^0(C,\omega_C)$ is a multiplicity-free representation of $\gm \subset \Aut(C)$
by Lemma \ref{L:multiplicityfree}, we can apply the Kempf-Morrison Criterion  
\cite[Proposition 2.3]{AFS-odd-stability} to prove semistability of $C$.
Namely, to prove that the $m^{th}$ Hilbert point of the canonically 
embedded balanced double $A_{2k+1}$-curve $C$ is semistable, it suffices to check that 
for every one-parameter subgroup 
$\rho\co \gm \rightarrow \SL(g)$ acting diagonally on the basis 
$\{x_1, \ldots, x_{k}, y_1, \ldots, y_{k}\}$ with integer weights $\lambda_1, \ldots, \lambda_{k}, \nu_1, \ldots, \nu_{k}$, 
there exists a monomial basis for $\HH^0(C, \omega_C^m)$ of non-positive $\rho$-weight.  
Explicitly, this means that we must exhibit a set $\B$ of $(2m-1)(2k-1)$ 
degree $m$ monomials in the variables $\{x_i,y_i\}_{i=1}^{k}$ 
with the properties that:
\begin{enumerate}
\item $\B$ maps to a basis of $\HH^0(C, \omega_C^m)$ 
via $\Sym^m \HH^0(C,\omega_C) \rightarrow \HH^0(C, \omega_C^m)$.
\item $\B$ has non-positive $\rho$-weight, i.e., if $\B=\{m_i\}_{i=1}^{(2m-1)(2k-1)}$, and 
$m_i=\prod_{j=1}^{k}x_j^{a_{ij}}y_j^{b_{ij}}$, then 
\[
\sum_{i=1}^{(2m-1)(2k-1)} \sum_{j=1}^k (a_{ij}\lambda_j+b_{ij}\nu_j) \leq 0.
\]
\end{enumerate}

\begin{theorem}\label{T:ribbon}
If $C \subset \PP \HH^0(C, \omega_C)$ is a canonically embedded balanced double $A_{2k+1}$-curve, then the Hilbert points $[C]_m$ are semistable for all $m \geq 2$.
\end{theorem}
As an immediate corollary of this result, we obtain a proof of Theorem \ref{T:main-trigonal} and hence of Theorem \ref{T:main-canonical}:
\begin{corollary}[Theorem \ref{T:main-trigonal}]\label{C:main-trigonal}
A general smooth trigonal curve of genus $g=2k$
embedded by the complete canonical linear system has a semistable $m^{th}$ Hilbert point for every $m \geq 2$.
\end{corollary}
\begin{proof}[Proof of Corollary]
By Proposition \ref{P:A-curve} the canonical embedding of the balanced double $A_{2k+1}$-curve $C$ lies on a balanced surface scroll in $\PP^{2k-1}$
in the divisor class $(3, k+1)$.
It follows that $C$ deforms flatly to a smooth curve in the class $(3, k+1)$ on the scroll. Such a curve is a smooth trigonal canonically embedded curve.
The semistability of a general deformation of $C$ follows from the openness of semistable locus. 
\end{proof}
\begin{proof}[Proof of Theorem \ref{T:ribbon}]
Recall from Lemma \ref{L:pluricanonical-bases} that 
\[
\HH^0(C,\omega_C^m)=\lspan \{ \omega_i \}_{i=0}^{m(k-1)}\oplus \lspan\{\eta_i \}_{i=0}^{m(k-1)}\oplus\lspan \{\chi_i\}_{i=-k(m-1)+1}^{k(m-1)-1}.
\]
Now, given a one-parameter subgroup $\rho$ as above, we will construct the requisite monomial basis $\B$ as a union
$$
\B=\B_{\omega} \cup \B_{\eta} \cup \B_{\chi},
$$
where $\B_{\omega}, \B_{\eta}$, and $\B_{\chi}$ are collections of degree 
$m$ monomials which map onto the bases of the subspaces spanned by 
$\{\omega_i\}_{i=0}^{m(k-1)}, \{\eta_i\}_{i=0}^{m(k-1)}$ and $\{\chi_i\}_{i=-k(m-1)+1}^{k(m-1)-1}$, respectively.

To construct $\B_{\omega}$ and $\B_{\eta}$, we use Kempf's 
proof of the stability of Hilbert points of a rational normal curve. 
More precisely, consider the component $C_0$ of $C$ with the uniformizer $s_0$ at $0\in C_0$.
Evidently, $\omega_{C}|_{C_0} \simeq \O_{\PP^1}(k-1)$. The restriction map 
$\HH^0(C, \omega_C) \rightarrow \HH^0(\PP^1, \O_{\PP^1}(k-1))$ identifies $\{x_i\}_{i=1}^{k}$ with a basis of 
$\HH^0(\PP^1, \O_{\PP^1}(k-1))$ given by $\{1,s_0,\dots, s_0^{k-1}\}$. Under this identification, the subspace $\lspan\{\omega_i\}_{i=0}^{m(k-1)}$ is identified with 
$\HH^0(\PP^1, \O_{\PP^1}(m(k-1)))$. 
Set $\lambda=\sum_{i=1}^k \lambda_i/k$. Given a one-parameter subgroup $\widetilde{\rho}\co \gm \ra \SL(k)$ acting 
on $(x_1,\dots, x_k)$ diagonally with weights $(\lambda_1-\lambda, \dots, \lambda_k-\lambda)$, Kempf's result on the semistability of a rational normal 
curve in $\PP^{k-1}$ \cite[Corollary 5.3]{kempf},
implies the existence of a monomial basis $\B_\omega$ of $\HH^0(\PP^1, \O_{\PP^1}(m(k-1)))$ with non-positive $\widetilde{\rho}$-weight.  
Under the above identification, $\B_{\omega}$ is a monomial basis of $\lspan\{\omega_i\}_{i=0}^{m(k-1)}$ of $\rho$-weight
at most $m(m(k-1)+1)\lambda$. Similarly, if $\nu=\sum_{i=1}^k\nu_i/k$, we deduce the existence 
of a monomial basis $\B_{\eta}$ of  $\lspan\{\eta_i \}_{i=0}^{m(k-1)}$ whose $\rho$-weight is at most $m(m(k-1)+1)\nu$. 
Since $\lambda+\nu=0$, it follows that the total $\rho$-weight of $\B_{\omega} \cup \B_{\eta}$ is non-positive.

Thus, to construct a monomial basis $\B$ of non-positive $\rho$-weight, it remains to construct a monomial basis $\B_{\chi}$ of non-positive $\rho$-weight
for the subspace $$\lspan\{\chi_i\}_{i=-k(m-1)-1}^{k(m-1)-1} \subset \HH^0(C, \omega_C^m).$$  
In Lemma \ref{L:basis}, proved below, we show the existence of such a basis. Thus, we obtain the desired monomial basis $\B$ and finish the proof.
\end{proof}

Note that if we define the {\em weighted degree} by $\deg(x_i)=i$ and $\deg(y_i)=-i$, then a set $\B_{\chi}$ of $2k(m-1)-1$ degree $m$ monomials in $\{x_1,\ldots, x_k, $ $y_1, \ldots, y_k\}$ maps to a basis of $\lspan\{\chi\}_{i=k(m-1)-1}^{k(m-1)+1}$ if and only if it satisfies the following two conditions:
\begin{enumerate}
\item Each monomial has both $x_i$ and $y_i$ terms,
\item Each weighted degree from $(m-1)k-1$ to $-(m-1)k+1$ occurs exactly once.
\end{enumerate}
We call such a set of monomials a {\em $\chi$-basis}. The following combinatorial lemma completes the proof of Theorem \ref{T:ribbon}.
\begin{lemma}\label{L:basis}
Suppose $\rho\co \gm\ra \SL(2k)$ is a one-parameter subgroup which 
acts on $\{x_1, \ldots, x_{k}, y_1, \ldots, y_{k}\}$ diagonally with integer weights $\lambda_1, \ldots, \lambda_{k}, \nu_1, \ldots, \nu_{k}$
satisfying $\sum_{i=1}^k (\lambda_i+\nu_i)=0$. Then there exists a $\chi$-basis with non-positive $\rho$-weight.
\end{lemma}

\begin{proof}[Proof of Lemma \ref{L:basis} for $m=2$]

Take the first $\chi$-basis to be
\begin{multline*}
\B_1:=\{ x_{k}y_1, x_{k-1}y_1, x_{k-1}y_2, x_{k-2}y_2, x_{k-2}y_3, \dots \\ \dots, x_{i}y_{k-i}, x_{i}y_{k-i-1}, 
\dots, x_2y_{k-1}, x_1y_{k-1}, x_1y_{k}\}
\end{multline*}
In this basis, all variables except $x_k$ and $y_k$ occur twice and $x_k$, $y_k$ occur once each. Thus 
$$w_{\rho}(\B_1)=2(\lambda_1+\cdots+\lambda_{k-1})+2(\nu_1+\cdots+\nu_{k-1})+\lambda_k+\nu_k=-(\lambda_k+\nu_k).$$
Take the second $\chi$-basis to be
\begin{align*}
\B_2:=\{ x_{k}y_1, x_{k}y_2, \dots, x_ky_i, \dots,x_{k}y_{k}, x_{k-1}y_{k}, x_{k-2}y_{k}, \dots, x_iy_k, \dots, x_1y_{k}\}.
\end{align*}
We have
$$w_{\rho}(\B_2)=(k-1)(\lambda_k+\nu_k).$$\\
For any one-parameter subgroup $\rho$, we must have either $\lambda_k+\nu_k \geq 0$ or $\lambda_k +\nu_k \leq 0$. Thus, either $\B_1$ or $\B_2$ gives a $\chi$-basis of non-positive weight.
\end{proof}

\begin{proof}[Proof of Lemma \ref{L:basis} for $m\geq 3$] 
We will prove the Lemma by exhibiting one collection of $\chi$-bases whose $\rho$-weights sum to a positive multiple of $\lambda_k +\nu_k$ and a collection of $\chi$-bases whose $\rho$-weights sum to a negative multiple of $\lambda_k+\nu_k$. Since, for any given one-parameter subgroup $\rho$,  we have either $\lambda_k+\nu_k \geq 0$ or $\lambda_k +\nu_k \leq 0$, it follows at once that one of our $\chi$-bases must have non-positive weight. We begin by writing down $\chi$-bases maximizing the occurrences of $x_k$ and $y_k$ while balancing the occurrences of the other variables.
Define $T_1$ as the set of degree $m$ monomials of the ideal
\begin{multline*}
x_{k}^{m-1}(y_1,\dots, y_{k-1}, y_{k}) + x_{k}^{m-2}y_{k} (y_1,\dots, y_{k-1}, y_{k}, x_1,\dots,x_{k-1})+\cdots \\
%x_{k}^{m-3}y_{k}^2 (y_1,\dots, y_{k-1}, y_{k}, x_1,\dots, x_{k-1}), \dots, \\
+x_{k}y_k^{m-2}(y_1,\dots, y_{k-1}, y_{k}, x_1,\dots, x_{k-1}) +  y_{k}^{m-1}(x_1,\dots, x_{k-1}).
\end{multline*}
The $\rho$-weight of $T_1$ is
{\small
$$\left(k(m-1)+(2k-1)\binom{m-1}{2}\right)(\lambda_k+\nu_k) + (m-1)(\lambda_1+\nu_1+ \cdots+\lambda_{k-1}+\nu_{k-1}).$$ 
}
Note that $T_1$ misses only the weighted degrees $$k(m-3), k(m-5), \dots,-k(m-5),-k(m-3).$$  For each $s=1, \dots, k-1$, define
$$\begin{aligned}
T_2(s) &: = \{x_{k}^{m-3}y_{k} (x_{k-s}x_{s}), x_{k}^{m-4}y_{k}^2 (x_{k-s}x_{s}), \dots, y_{k}^{m-2}(x_{k-s}x_{s})\}\\
T'_2(s) &:= \{y_{k}^{m-3}x_{k} (y_{k-s}y_{s}), y_{k}^{m-4}x_{k}^2 (y_{k-s}y_{s}), \dots, x_{k}^{m-2}(y_{k-s}y_{s})\}
\end{aligned}$$
For each $s$, the sets $T_1 \cup T_2(s)$ and $T_1 \cup T'_2(s)$ are $\chi$-bases. Using the relation $\sum_{i=1}^{k}(\lambda_i +\nu_i)=0$, one sees at once that the sum 
of the $\rho$-weights of such bases, as
$s$ ranges from $1$ to $k-1$, is a {\em positive multiple} of $(\lambda_k+\nu_k)$.

We now write down bases minimizing the occurrences of $x_k$ and $y_k$.  
We handle the case when $k$ is even and odd separately.

\subsubsection*{Case of even $k$:} If $k=2\ell$, we define the following set of monomials where the weighted degrees range from  $k(m-1)-1$ to $m$:
\begin{equation*}
S_1 := \left\{
\begin{array}{l}   \left. 
\begin{aligned}
&x_{k}^{m-1}y_1, && x_{k}^{m-2}x_{k-1}y_1,& &\dots,&& x_{k-1}^{m-1}y_1, \\
&x_{k-1}^{m-1}y_2, && x_{k-1}^{m-2}x_{k-2}y_2,& &\dots,&& x_{k-2}^{m-1}y_2, \\
&& \vdots \\
&x_{\ell+2}^{m-1}y_{\ell-1}&& x_{\ell+2}^{m-2}x_{\ell+1}y_{\ell-1},& &\dots,&& x_{\ell+1}^{m-1}y_{\ell-1}
\end{aligned}\right\} \begin{array}{l} \text{$m$ terms in } \\ \text{each of the} \\ \text{($\ell-1$) rows} \end{array}
 \\ 
 \\
 \left.\begin{aligned}
&x_{\ell+1}^{m-1}y_{\ell},&& x_{\ell+1}^{m-2}x_{\ell}y_{\ell}, &&\dots, &&x_{\ell+1}^{2}x_{\ell}^{m-3}y_{\ell}, \\
&x_{\ell}^{m-1}y_{\ell-1},&& x_{\ell}^{m-2}x_{\ell-1}y_{\ell-1}, &&\dots,&& x_{\ell}^{2}x_{\ell-1}^{m-3}y_{\ell-1}, \\
&&\vdots  \\
&x_{2}^{m-1}y_1,&& x_{2}^{m-2}x_1y_1, &&\dots, &&x_{2}^{2}x_{1}^{m-3}y_1
\end{aligned}\right\} \begin{array}{l} \text{$(m-2)$ terms}\\ \text{in each of the} \\ \text{$\ell$ rows} \end{array}
\end{array}
\right.
\end{equation*}
Let $\iota$ be the involution of the set $\{  x_i, y_i \}_{i=1}^k$ exchanging $x_i$ and $y_i$.
In the set $S_1 \cup \iota(S_1)$, the variables $x_k$ and $y_k$ occur $(m^2-m)-\binom{m}{2}$ times,  
$x_{\ell+1}$ and $y_{\ell+1}$ occur $(m^2-m)-1$ times, $x_{\ell}$ and $y_{\ell}$ occur $(m^2-m)-m$ times, 
and $x_1$ and $y_1$ occur $m^2-m-(\binom{m}{2}-1)$ times while all of the other variables occur $m^2-m$ times. To complete $S_1 \cup \iota(S_1)$ to a $\chi$-basis, we define, for each $s=1, \ldots, k-1$, the following set of monomials where the weighted degrees
range from $m-1$ to $1-m$:
\begin{equation*}
S_2(s):= \left\{
\begin{aligned}
&x_{\ell+1}y_{\ell}x_1^{m-2} \\
&x_{\ell}y_{\ell}(x_sy_s)^ix_1^{m-2i-2},  &&\quad \text{for $0\leq 2i \leq m-2$}, \\
&x_{\ell}y_{\ell}(x_sy_s)^{i}y_1^{m-2i-2},  &&\quad \text{for $0\leq 2i < m-2$},  \\
&(x_{k}y_{s}y_{k-s})(x_sy_s)^ix_1^{m-2i-3},  &&\quad \text{for $0\leq 2i \leq m-3$}, \\
&(x_{k}y_{s}y_{k-s})(x_sy_s)^iy_1^{m-2i-3},  &&\quad \text{for $0\leq 2i < m-3$}, \\
&y_{\ell+1}x_\ell y_{1}^{m-2}.
\end{aligned}
 \right\}
\end{equation*}
%Note that in the set $S_2(s)$, if we disregard the weights arising from powers of $(x_sy_s)$ and $(x_ky_sy_{k-s})$, 
%the variables $x_{\ell+1}$ and $y_{\ell+1}$ occur $1$ time,
%while $x_1$ and $y_1$ occur $\binom{m}{2}-1$ times. 
%Let $S'_2(s)$ be the set obtained from $S_2(s)$ by replacing each occurrence of $(x_ky_sy_{k-s})$ by $(x_s x_{k-s} y_k)$. 
For each $s = 1, \ldots, k-1$, the sets $S_1 \cup \iota(S_1) \cup S_2(s)$ and $S_1 \cup \iota(S_1) \cup \iota(S_2(s))$ are 
$\chi$-bases.
We compute that in the union 
\begin{equation*}
\bigcup_{s=1}^k \left(S_1 \cup \iota(S_1) \cup S_2(s)\right) \cup \left(S_1 \cup \iota(S_1) \cup \iota(S_2(s))\right)
\end{equation*}
of $2(k-1)$ $\chi$-bases the variables $x_k$ and $y_k$ each occurs 
$$2(k-1)(m^2-m)-(k-1)(m^2-2m+2)$$ 
times while all of the other variables occur 
$$2(k-1)(m^2-m)+(m-2)(m-1)$$ times. 

Using the relation $\sum_{i=1}^{k}(\lambda_i +\nu_i)=0$, we conclude that the sum of the $\rho$-weights of all such $\chi$-bases is a {\em negative multiple} of 
$(\lambda_k + \nu_k)$.

\subsubsection*{Case of odd $k$:} If $k=2\ell+1$ is odd, $\chi$-bases of non-positive $\rho$-weight can be constructed analogously to the case when $k$ is even. 
For the reader's convenience, we spell out the details.  
We define of the following set of monomials where the weighted degrees range from  $k(m-1)-1$ to $m-1$:
\begin{equation*}
S_1 := \left\{
\begin{array}{l}   \left. 
\begin{aligned}
&x_{k}^{m-1}y_1, && x_{k}^{m-2}x_{k-1}y_1,& &\dots,&& x_{k-1}^{m-1}y_1, \\
&& \vdots \\
&x_{\ell+3}^{m-1}y_{\ell-1}&& x_{\ell+3}^{m-2}x_{\ell+2}y_{\ell-1},& &\dots,&& x_{\ell+2}^{m-1}y_{\ell-1}
\end{aligned}\right\} \begin{array}{l} \text{$m$ terms in} \\ \text{each of the} \\ \text{($\ell-1$) rows} \end{array}
 \\ 
 \\
 \left.\begin{aligned}
&x_{\ell+2}^{m-1}y_{\ell},&& x_{\ell+2}^{m-2}x_{\ell+1}y_{\ell}, &&\dots, &&x_{\ell+2}^{2}x_{\ell+1}^{m-3}y_{\ell}, \\
&x_{\ell+1}^{m-2}y_{\ell-1},&& x_{\ell+1}^{m-1}x_{\ell}y_{\ell}, &&\dots,&& x_{\ell+1}^{2}x_{\ell}^{m-3}y_{\ell-1}, \\
&&\vdots  \\
&x_{3}^{m-1}y_1,&& x_{3}^{m-2}x_2y_1, &&\dots, &&x_{3}^{2}x_{2}^{m-3}y_1
\end{aligned}\right\} \begin{array}{l} \text{$(m-2)$ terms}  \\ \text{in each of the} \\ \text{$\ell$ rows} \end{array}
 \\ 
 \\
 \begin{aligned}
x_{\ell+2}y_{\ell}x_2^{m-2}, \\
\end{aligned} %\} \text{$1$ term}
\\
 \\ 
 \begin{aligned}
x_{\ell+1}y_{\ell}x_2^{m-2}, x_{\ell+1}y_{\ell}x_2^{m-3}x_1, \ldots, x_{\ell+1}y_{\ell}x_2 x_1^{m-3}, x_{\ell+1}y_{\ell}x_1^{m-2}
\end{aligned} 

\end{array}
\right.
\end{equation*}
Let $\iota$ be the involution exchanging $x_i$ and $y_i$.
In the set of monomials $S_1 \cup \iota(S_1)$, the variables $x_k$ and $y_k$ occur $\binom{m}{2}$ times, $x_{\ell+1}$ and $y_{\ell+1}$ occur $m^2-m-(m-1)$ times, 
and $x_1$ and $y_1$ occur $m^2-m - \binom{m-1}{2}$ times, while all of the other variables occur $m^2-m$ times.  Finally, for each $s=1, \ldots, k-1$, we define the following set of monomials where the weighted degrees range from $m-2$ to $2-m$:
\begin{equation*}
S_2(s):= \left\{
\begin{aligned}
&x_{\ell+1}y_{\ell+1}(x_sy_s)^ix_1^{m-2-2i}, &&\quad \text{for $0\leq 2i \leq m-2$}, \\
&x_{\ell+1}y_{\ell+1}(x_sy_s)^iy_1^{m-2-2i}, &&\quad \text{for $0\leq 2i < m-2$}, \\
&(x_k y_s y_{k-s})(x_sy_s)^ix_1^{m-3-2i}, &&\quad \text{for $0\leq 2i \leq m-3$}, \\
&(x_k y_s y_{k-s})(x_sy_s)^iy_1^{m-3-2i}, &&\quad \text{for $0\leq 2i < m-3$} 
\end{aligned}
\right\}
\end{equation*}
For each $s=1,\dots, k-1$, the sets 
$S_1 \cup \iota(S_1)\cup S_2(s)$ and $S_1 \cup \iota(S_1) \cup \iota(S_2(s))$ are $\chi$-bases.
We compute that in the union 
\begin{equation*}
\bigcup_{s=1}^k \left(S_1 \cup \iota(S_1) \cup S_2(s)\right) \cup \left(S_1 \cup \iota(S_1) \cup \iota(S_2(s))\right)
\end{equation*}
of $2(k-1)$ $\chi$-bases the variables $x_k$ and $y_k$ each occurs 
$$2(k-1)\binom{m}{2}+2(k-1)(m-2)$$ 
times while all of the other variables occur 
$$2(k-1)(m^2-m)+(m-2)(m-1)$$ times. 

Using the relation $\sum_{i=1}^{k}(\lambda_i +\nu_i)=0$, we conclude that
the total $\rho$-weight of these 
$\chi$-bases is a {\em negative multiple} of $(\lambda_k + \nu_k)$ and we're done.
\end{proof}

\section{Non-semistability results}
\label{S:bielliptic}

The generic semistability results of Theorem \ref{T:main-canonical} and \cite[Theorem 1.2]{AFS-odd-stability} 
raise a natural question of whether Hilbert points of 
smooth canonically embedded curves can at all be non-semistable. An indirect way to see that the answer is affirmative is as follows. 
Denote by $\overline{H}_{g,1}^{\, m}$ the closure of the locus of $m^{th}$ Hilbert points of smooth canonical curves. Next, it is proved in \cite[Section 5]{hassett-hyeon_flip} that an application of
Grothendieck-Riemann-Roch formula allows to write 
the polarization on the GIT quotient $\overline{H}_{g,1}^{\, m}\gitq \SL(g)$ as a linear combination
\begin{multline}\label{E:polarization}
(m(m-1)(4g+2)-(m-1)(g-1)+1)\lambda-\frac{gm(m-1)}{2}\delta \\ \sim \left[8+\frac{4}{g}-\frac{2(g-1)}{gm}+\frac{2}{gm(m-1)}\right]\lambda-\delta
\end{multline}
of a tautological divisor $\lambda$ (the first Chern class of the Hodge bundle) and the boundary divisor $\delta$ (at least on the locus parameterizing curves with mild singularities). 
By generalizing the proof of \cite[Proposition 4.3]{CH}, 
it is not too difficult to see that if $B\ra \Mg{g}$ is a family of stable curves whose general fiber is canonically embedded and the slope $(\delta\cdot B)/(\lambda\cdot B)$
is greater than $(8+\frac{4}{g})-\frac{2(g-1)}{gm}+\frac{2}{gm(m-1)}$, 
then every curve in $B$ (with a well-defined $m^{th}$ Hilbert point) 
must have a non-semistable $m^{th}$ Hilbert point.

Two observations now lead to a candidate for a non-semistable canonically embedded curve. The first is that
$(8+\frac{4}{g})-\frac{2(g-1)}{gm}+\frac{2}{gm(m-1)}\leq 8$ for $g\geq 2m+1+1/(m-1)$. The second is that there are families 
of bielliptic curves of slope $8$ (such can be constructed by taking a double cover of a trivial family of elliptic curves). 
In the following result, we establish that bielliptic curves indeed become non-semistable 
for small values of $m$, and show that generic bielliptic curves are semistable for $m$ large enough.

%\subsection{Bielliptic curves}
\begin{theorem}
A smooth bielliptic curve of genus $g$ has non-semistable
$m^{th}$ Hilbert point for all $m \leq (g-3)/2$. A general bielliptic curve of odd genus $g=2k+1$ has 
semistable $m^{th}$ Hilbert point for $m\geq (g-1)/2$.
\end{theorem}
\begin{proof}
Let $C$ be a bielliptic canonical curve. Then $C$ is a quadric section of a projective cone over an elliptic curve $E\subset \PP^{g-2}$ 
embedded by a complete linear system of degree $g-1$. Choose projective coordinates $[x_0:\ldots: x_{g-1}]$.
Suppose that the vertex of the cone has coordinates $[0:0:\ldots:0:1]$.
Let $\rho$ be the one-parameter subgroup of $\SL(g)$ acting with weights $(-1,-1,\dots, -1, g-1)$.
There are 
\[
s_m:=h^0(\PP^{g-2}, \O_{\PP^{g-2}}(m))-h^0(E,\O_E(m))=\binom{g-2+m}{m}-m(g-1)
\] degree $m$ hypersurfaces containing $E$.
Thus $$\dim \HH^0(C,\I_C(m)) \cap (x_0,x_2,\dots,x_{g-2})^m=s_m$$ 
and so there are at most 
$$h^0(C,\O_C(m))-s_m=m(g-1)$$ elements in $\HH^0(C,\O_C(m))$ of $\rho$-weight $(-m)$. The remaining $(m-1)(g-1)$ elements in $\HH^0(C,\O_C(m))$
have $\rho$-weight at least $g-m$. Thus the $\rho$-weight of any monomial basis of $\HH^0(C,\O_C(m))$ is at least 
\begin{equation}\label{E:bielliptic-weight}
(m-1)(g-1)(g-m)-m(m(g-1))%=(g-1)((m-1)(m-g)+m^2)
=(g-1)((g+1)m-2m^2-g).
\end{equation}
If $m\leq (g-3)/2$, then \eqref{E:bielliptic-weight} is positive, and so $C$ has a non-semistable $m^{th}$ Hilbert point.

To prove the generic semistability of bielliptic curves in the range $m\geq (g-1)/2$, we recall \cite[Theorem 4.12]{AFS-odd-stability}
which shows that the odd genus $g$ canonically embedded rosary has a semistable $m^{th}$ Hilbert point if and only if $g\leq 2m+1$.
It remains to observe that the canonically embedded rosary deforms to a canonically embedded smooth bielliptic curve in the Hilbert scheme
of canonically embedded curves. This is accomplished in Lemma \ref{L:rosary} below. 
\end{proof}

\begin{lemma}
\label{L:rosary}
The canonically embedded rosary 
deforms flatly to a canonically embedded bielliptic curve.
\end{lemma}
\begin{proof} Let $C$ be 
the rosary of genus $g=2k+1$ introduced by Hassett and Hyeon \cite[Section 8.1]{hassett-hyeon_flip}.
We use the notation of \cite[Section 3.2]{AFS-odd-stability}.
% and in particular a basis of $\HH^0(C,\omega_C)$ 
%described in \cite[Lemma 3.6]{AFS-odd-stability}. 

Consider $\PP^{g-2}$ with projective coordinates $[x_0:\ldots:x_{g-2}]$ and define $E\subset\PP^{g-2}$ to be the union of $g-1$ lines $L_i : \{x_{i+1}=\dots=x_{i+g-3}=0\}$,
for $i=0,\dots, g-2$
(we use the convention that $x_{i+g-1}=x_{i}$). Then $E$ is a nodal curve of arithmetic genus $1$. Since $\HH^1(C,\O_C(1))=0$, we can deform $E$ in a flat family 
to a smooth elliptic curve by \cite[p.83]{Kol}. Using the basis $(\eta, \omega_0, \dots, \omega_{g-2})$ 
of $\HH^0(C,\omega_C)$ described in \cite[Lemma 3.6]{AFS-odd-stability},
we observe that the rosary $C$ is cut out by the quadric 
$$y^2=x_0x_1+x_1x_2+\cdots+x_{g-2}x_0$$
%$$y^2=\omega_0\omega_1+\omega_1\omega_2+\cdots+\omega_{g-2}\omega_0$$
on the projective cone over $E$ in $\PP^{g-1}$.  Since $E$ deforms to a smooth elliptic curve, it follows that $C$ deforms to a smooth bielliptic curve.
\end{proof}

%\subsection{Trigonal curves of higher Maroni invariant}

\begin{remark}[Trigonal curves of higher Maroni invariant] Theorem \ref{T:main-trigonal} shows that the general trigonal curve with Maroni invariant $0$ has a semistable $m^{th}$ Hilbert point for all $m\geq 2$. 
In joint work of the second author with Jensen, it is shown that every trigonal curve with Maroni invariant $0$ has a semistable $2^{nd}$ Hilbert point and every trigonal
curve with a positive Maroni invariant has a non-semistable $2^{nd}$ Hilbert point \cite{fedorchuk-jensen}. In view of the asymptotic stability of the canonically embedded 
curves \cite{mumford-stability}, this result suggests that {\em every} smooth trigonal curve of Maroni invariant $0$ has a semistable $m^{th}$ Hilbert
point for every $m\geq 2$. One also expects that for a general smooth trigonal curve of positive Maroni invariant 
already the third Hilbert point is semistable. Indeed, Equation \ref{E:polarization} shows that the polarization 
on $\overline{H}_{g,1}^{\, 3}\gitq \SL(g)$ is a multiple of 
\begin{equation}\label{E:slope-3}
\left(\frac{22}{3}+\frac{5}{g}\right)\lambda-\delta.
\end{equation}
On the other hand, the maximal possible slope for a family of generically smooth trigonal curves of genus $g$
 is $36(g+1)/(5g+1)$ by \cite{stankova}. 
We note that $$36(g+1)/(5g+1)\leq \left(\frac{22}{3}+\frac{5}{g}\right)$$ whenever $(g-3)(2g-5)\geq 0$. Thus we expect that the $3^{rd}$ Hilbert 
point of a genus $g\geq 4$ canonically embedded trigonal curve is stable. 
%Proving this will likely require new ideas.
\end{remark}

\bibliographystyle{alpha}
\bibliography{BibHK}

\end{document}